\documentclass[11pt]{article}
\usepackage{amsmath, amssymb, amsthm, amsfonts}
\usepackage{indentfirst}
\usepackage[mathscr]{euscript}
\usepackage{amscd}
\usepackage{color}
\usepackage[pdftex]{graphicx}
\usepackage[left=3cm, right=3cm, top=3cm, bottom=3cm, a4paper]{geometry}

\theoremstyle{plain}
\newtheorem{theorem}{Theorem}\numberwithin{theorem}{section}
\newtheorem{lemma}{Lemma}\numberwithin{lemma}{section}
\numberwithin{proposition}{section}
\newtheorem{corollary}{Corollary}\numberwithin{corollary}{section}
\theoremstyle{definition}
\newtheorem{definition}{Definition}\numberwithin{definition}{section}
\theoremstyle{remark}
\newtheorem{remark}{Remark}\numberwithin{remark}{section}

\newcommand{\R}{\mathbb{R}}

\newcommand{\J}{\mathbb{J}}
\newcommand{\F}{{}_1F_2}

\newcommand{\bx}{\mathbf{x}}

\title{An extension of positivity for integrals of Bessel functions and Buhmann's radial basis functions}

\author{Yong-Kum Cho, Seok-Young Chung and Hera Yun\footnote{E-mails: ykcho@cau.ac.kr,\hskip.2cm
sychung@cau.ac.kr, \hskip.2cm herayun06@gmail.com}}

\date{\small Department of Mathematics, College of Natural Science, Chung-Ang University,\\
84 Heukseok-Ro, Dongjak-Gu, Seoul 06974, Korea}

\begin{document}

\maketitle

\bigskip

%\noindent
{\bf Abstract.} As to the Bessel integrals of type
\begin{equation*}
\int_0^x \left(x^\mu-t^\mu\right)^\lambda t^\alpha J_\beta(t)dt\qquad(x>0),
\end{equation*}
we improve known positivity results by making use of new positivity criteria
for $\F$ and ${}_2F_3$ generalized hypergeometric functions. As an application, we extend Buhmann's
class of compactly supported radial basis functions.

\bigskip

%\noindent

\bigskip

%\noindent

\section{Introduction}
We consider the problem of determining positivity of the integrals
\begin{equation}\label{A1}
\int_0^x \left(x^\mu-t^\mu\right)^\lambda t^\alpha J_\beta(t)dt\qquad(x>0)
\end{equation}
for appropriate values of parameters $\,\mu,\,\lambda,\,\alpha,\,\beta,\,$ where $J_\beta$ stands for the
Bessel function of order $\beta$. For the sake of convergence and practical applications, it is common to assume
$\,\mu>0,\,\lambda\ge 0,\,\beta>-1,\,\alpha+\beta+1>0.\,$

Owing to various applications, the problem has been studied by many authors over a long period of time
and we refer to Askey \cite{As1} and Gasper \cite{Ga1} for historical backgrounds. Of our primary concern is the result of
Misiewicz and Richards which states in a simplified version as follows.

\begin{itemize}
\item[] {\bf Theorem A.} (Misiewicz and Richards \cite{MR}) {\it Let $\mathcal{A}$ be the set of parameters $(\beta, \alpha)$ defined by
$$\mathcal{A} = \left\{\beta>-\frac 12\,,\,\,-\beta-1<\alpha\le\min\left(\beta,\,\frac 32\right)\right\}.$$
For $\,0<\mu\le 1\,$ and $\,\lambda\ge 1,\,$ if $\,(\beta, \alpha)\in\mathcal{A},\,$ then
\begin{equation*}
\int_0^x \left(x^\mu-t^\mu\right)^\lambda t^\alpha J_\beta(t)dt>0 \qquad(x>0).
\end{equation*}}
\end{itemize}

An additional range of parameters $\,\alpha, \beta\,$ is also available.
In fact, if $j_{\beta, 2}$ denotes the second positive zero of $J_\beta$ and
$\alpha_*(\beta)$ the solution of
\begin{equation}
\int_0^{j_{\beta, 2}} t^{\alpha_*(\beta)} J_\beta(t) dt = 0
\end{equation}
for each $\beta$, then Misiewicz and Richards pointed out the above positivity holds
true for $\,-1<\beta<\frac 12,\,-\beta-1<\alpha<\alpha_*(\beta).$ As it is described in detail by Askey \cite{As2}, however,
the explicit nature of $\alpha_*(\beta)$ is still unknown and we shall exclude this range throughout.

In the special case $\,\alpha = \frac 12,\,\beta=-\frac 12,\,$ the integrals of \eqref{A1} reduce to
the Fourier cosine transforms for which Kuttner proved its positivity:

\begin{itemize}
\item[] {\bf Theorem B.} ( Kuttner \cite{K}) {\it For $\,0<\mu\le 1\,$ and $\,\lambda\ge 1,\,$
\begin{equation*}
\int_0^x \left(x^\mu-t^\mu\right)^\lambda \cos t\, dt>0 \qquad(x>0).
\end{equation*}}
\end{itemize}

\begin{figure}[!h]
 \centering
 \includegraphics[width=280pt, height=220pt]{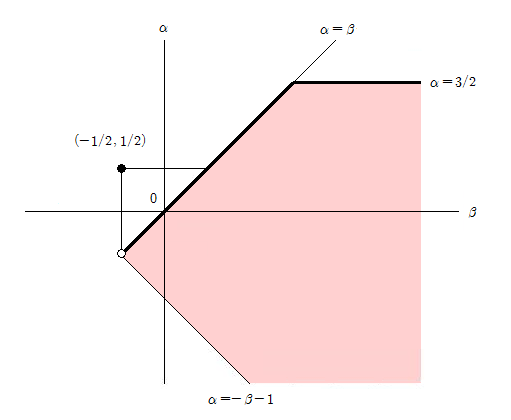}
 \caption{The positivity regions of Misiewicz-Richards (pink) and Kuttner (black dot)}
\label{Fig1}
\end{figure}

The combined positivity region is depicted in Figure \ref{Fig1}.

The main purpose of the present paper is to improve Theorem A and Theorem B by extending
positivity regions for $\,0<\mu\le 1,\,\lambda\ge 1\,$ as
well as by providing a positivity region for unrestricted $\,\mu>0,\,\lambda\ge 0.$

As an application of our results, we shall improve in several directions the range of positive definiteness for Buhmann's
class of compactly supported radial basis functions \cite{Bu} which are of considerable interest in the theory of
approximations and interpolations.

\section{Positivity in the unrestricted case}
For $\,\alpha, \beta\,$ satisfying $\,\beta>-1,\,\alpha+\beta+1>0,\,$ if we put
\begin{equation}
\Phi(x) = \F\left(\begin{array}{c} \frac{\alpha+\beta+1}{2}\\\beta+1,\,\frac{\alpha+\beta+3}{2}
\end{array}\biggl| -\frac{\,x^2}{4}\right),
\end{equation}
then it is simple to evaluate by integrating termwise or by parts
\begin{align}
&\int_0^x t^\alpha J_\beta(t) dt = \frac{x^{\alpha+\beta+1}}{2^\beta\Gamma(\beta+1)(\alpha+\beta+1)}\,\Phi(x),\\
&\int_0^x \left(x^\mu-t^\mu\right)^\lambda t^\alpha J_\beta(t)dt\nonumber\\
&\qquad=\frac{\mu\lambda x^{\mu\lambda +\alpha+\beta+1}}{2^\beta\Gamma(\beta+1)(\alpha+\beta+1)}
\int_0^1\Phi(xt) (1-t^\mu)^{\lambda-1} t^{\mu+\alpha+\beta}\,dt
\end{align}
for $\,\mu>0, \,\lambda>0\,$ and $\,x>0.$ Therefore positivity of \eqref{A1} would follow once kernel $\Phi$ were shown to be positive in the case
$\,\lambda=0\,$ or nonnegative in the case $\,\lambda>0.$

To investigate the sign of $\F$ generalized hypergeometric function $\Phi$, we shall make use of the
following general criterion recently established by the authors \cite{CY}, which will be applied subsequently in other
occasions as well.

As it is standard, the Newton diagram associated to a finite set of planar points
$\,\bigl\{\left(\alpha_i,\,\beta_i\right) : i= 1, \cdots, N\bigr\}\,$
refers to the closed convex hull containing
$$\bigcup_{i=1}^N\,\bigl\{(x, y)\in\R^2 : x\ge \alpha_i,\,\,y\ge \beta_i\bigr\}\,.$$

\begin{lemma}\label{lemma1}{\rm(Cho and Yun, \cite{CY})}
For $\,a>0,\,b>0,\,c>0,\,$ put
$$\phi(x)=\F\left(a\,; b, c\,; -\frac{\,x^2}{4}\right)\quad(x>0).$$
\begin{itemize}
\item[\rm(i)] If $\,\phi\ge 0,\,$ then necessarily
$\, b>a,\,\,c>a,\,\,b+c\ge 3a+\frac 12\,.$
\item[\rm(i)] Let $P_a$ denote the Newton diagram associated to
$$\,\Lambda = \left\{\left(a+\frac 12,\,2a\right),\,\left(2a,\,a+\frac 12\right)\right\}.\,$$
If $\,(b, c)\in P_a\,,$ then $\,\phi\ge 0\,$ and strict positivity holds unless $(b, c)\in\Lambda.$
\end{itemize}
\end{lemma}

\smallskip

For the sake of presenting this paper in a self-contained way,
we shall give a simplified proof in the appendix.
Keeping in mind that the line segment joining two point of $\Lambda$ is given by
$\,c= 3a + 1/2 -b\,$ in the $\,(b, c)$-plane, it is straightforward to obtain the
range for the positivity or nonnegativity of $\Phi$ by
implementing Lemma \ref{lemma1}.

\smallskip
\begin{theorem}\label{theorem1}
Let $\mathcal{R}$ be the set of parameters $(\beta, \alpha)$ defined by
$$\mathcal{R} = \left\{\beta>-1,\,\,-\beta-1<\alpha\le 0\right\}\cup\left\{\beta>0, \,\,0<\alpha\le\min \left(\beta, \,\frac 12\right)\right\}.$$
For $\,\mu>0,\,\lambda\ge 0\,$ and $\,(\beta, \alpha)\in\mathcal{R},\,$ we have
\begin{equation*}
\int_0^x \left(x^\mu-t^\mu\right)^\lambda t^\alpha J_\beta(t)dt>0 \qquad(x>0)
\end{equation*}
unless $\,\lambda =0, \,\alpha=\beta=1/2.\,$ In the exceptional case, it reduces to
\begin{equation*}
\int_0^x J_{\frac 12}(t)\sqrt t\,dt = \frac{2\sqrt 2}{\sqrt\pi}\,\sin^2\left(\frac x2\right)\ge 0.
\end{equation*}
\end{theorem}

\begin{figure}[!h]
 \centering
 \includegraphics[width=280pt, height=220pt]{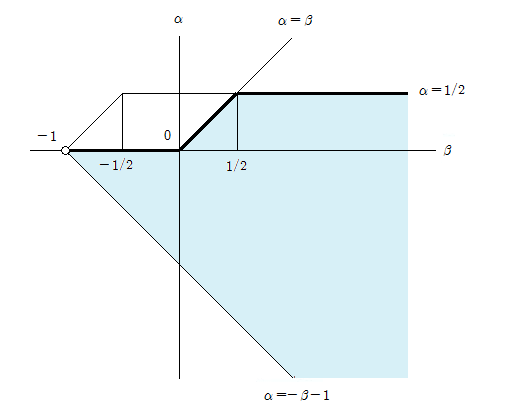}
 \caption{The positivity region in the unrestricted case $\mu>0,\,\lambda\ge 0.$}
\label{Fig2}
\end{figure}

\begin{remark} Geometrically, $\mathcal{R}$ represents an infinite polygonal region
depicted as in Figure \ref{Fig2}. In the case $\,\lambda=0,\,$ it follows from an inspection on Lemma \ref{lemma1} that
the necessity region for nonnegativity is given by
$\,\left\{\beta>-1, \,-\beta-1<\alpha\le 1/2,\,\alpha<\beta+1\right\}\,$
so that our result does not cover the parallelogram defined by
$\,\left\{0<\alpha\le 1/2,\,\beta<\alpha<\beta+1\right\}.$
\end{remark}

\begin{proof}
For the positivity of $\Phi$, we write
$\,A=(\alpha+\beta+1)/2\,$ and apply Lemma \ref{lemma1} with $\,a=A, \,b=A+1,\,c=\beta+1.\,$
For $\,0<A<1/2,\,$ $\Phi$ is positive when $\,\beta+1\ge 2A,\,$ that is,
$\,-\beta-1<\alpha<-\beta,\,\,\alpha\le 0.$
For $\,A\ge 1/2,\,$ $\Phi$ is positive when $\,\beta+1\ge \max(2A-1/2,\,A+1/2),\,$
that is, $\,-\beta\le\alpha\le \min(\beta,\,1/2).\,$
Combining, we obtain the stated region of positivity.
\end{proof}

In the special case $\,\mu=2,\,\lambda>0,\,$ positivity region $\mathcal{R}$ of Theorem \ref{theorem1} can be
improved considerably. As a matter of fact, if we observe
\begin{align}
\int_0^x \left(x^2-t^2\right)^\lambda t^\alpha J_\beta(t)dt =
\frac{B\left(\lambda+1, \frac{\alpha+\beta+1}{2}\right)\,x^{2\lambda+\alpha+\beta+1}}{2^{\beta+1}
\Gamma(\beta+1)}\,\Psi(x),
\end{align}
where $B$ stands for the Euler's beta function and
\begin{equation}
\Psi(x) = \F\left(\begin{array}{c} \frac{\alpha+\beta+1}{2}\\\beta+1,\,\lambda+1+\frac{\alpha+\beta+1}{2}
\end{array}\biggl| -\frac{\,x^2}{4}\right),
\end{equation}
then it is routine to deduce from Lemma \ref{lemma1} the following result.

\smallskip

\begin{theorem}\label{theorem2}
Let $\mathcal{S}$ be the set of parameters $(\beta, \alpha)$ defined by
$$\mathcal{S} = \left\{\beta>-1,\,\,-\beta-1<\alpha\le 0\right\}\cup
\left\{\beta>0,\,\,\alpha\le\min \left(\beta, \,\lambda+\frac 12\right)\right\}.$$
If $\,\lambda>0\,$ and $\,(\beta, \alpha)\in\mathcal{S},\,$ then
\begin{equation*}
\int_0^x \left(x^2-t^2\right)^\lambda t^\alpha J_\beta(t)dt>0 \qquad(x>0)
\end{equation*}
unless $\,\alpha=\beta=\lambda+1/2\,$ for which the integral reduces to
\begin{equation*}
\int_0^x \left(x^2-t^2\right)^\lambda t^{\lambda+\frac 12} J_{\lambda+\frac 12}(t)dt
=\frac{\sqrt\pi\,\Gamma(\lambda+1)(2x^2)^{\lambda+\frac 12}}{2} \,J_{\lambda+\frac 12}^2\left(\frac x2\right)\ge 0.
\end{equation*}
\end{theorem}

\smallskip

\begin{remark}
In \cite{Ga1}, Gasper also obtained a positivity region in this case.
Our result, however, is an improvement in that
the triangle with vertices $\,(0, \,0),\,(\lambda+1/2,\,0),\,(\lambda+1/2,\,\lambda+1/2)\,$ is missing in Gasper's positivity region.
\end{remark}

\section{Positivity of $_2F_3$ hypergeometric functions}
While Newton diagrams give positivity regions of $\F$ hypergeometric functions,
it appears that there are no such criteria available for $_2F_3$ hypergeometric functions.
Our purpose here is to develop some criteria of positivity, which will be exploited later on.

As usual, we shall use Pochhammer's notation to denote
$$(\alpha)_k = \alpha(\alpha+1)\cdots (\alpha+k-1),\quad (\alpha)_0 =1$$
for any real number $\alpha$ and positive integer $k$. We refer to Bailey \cite{Ba}, Luke \cite{L} for
definitions and basic properties of generalized hypergeometric functions.

A basic feature on positivity is the following.

\smallskip

\begin{lemma}\label{lemma3}
For positive real numbers $\,a, b, c, d, e,\,$ suppose that
\begin{equation*}
{}_2F_{3} \left(\begin{array}{c} a, \,b \\
c,\,d,\,e \end{array}\biggl| -\frac{\,x^2}{4}\right)
>0\qquad(x>0).
\end{equation*}
Then for any $\,\delta\ge 0,\,\gamma\ge 0,\,\epsilon\ge 0,\,$ we also have
\begin{equation*}
{}_2F_{3} \left(\begin{array}{c} a, \,b \\
c+\delta,\,d+\gamma,\,e+\epsilon \end{array}\biggl| -\frac{\,x^2}{4}\right)
>0\qquad(x>0).
\end{equation*}
\end{lemma}

\begin{proof}
Assuming $\,\delta>0,\,$ we have
\begin{align*}
&{}_2F_{3} \left(\begin{array}{c} a, \,b \\
c+\delta,\,d,\,e \end{array}\biggl| -\frac{\,x^2}{4}\right)\\
&= \frac{2}{B(c, \delta)} \int_0^1 {}_2F_{3} \left(\begin{array}{c} a, \,b \\
c,\,d,\,e \end{array}\biggl| -\frac{\,x^2t^2}{4}\right) (1-t^2)^{\delta-1} t^{2c-1}\,dt>0
\end{align*}
and the other cases follow in the same manner or by symmetry.
\end{proof}

\smallskip

We next deal with $_2F_3$ hypergeometric functions of the form
\begin{align}\label{W1}
\Omega(x) &={}_2F_{3} \left(\begin{array}{c} a, \,a+\frac 12\\
c+1,\,a+b,\,a+b + \frac 12\end{array}\biggl| -\frac{\,x^2}{4}\right)\nonumber\\
&= \frac{1}{B(2a,\,2b)}\int_0^1
{}_0F_1\left(c+1\,;-\frac{\,x^2t^2}{4}\right)(1- t)^{2b-1} t^{2a-1}\, dt
\end{align}
with parameters satisfying $\,a>0,\,b>0,\,c>-1.$

We apply Gasper's
{\it sums of squares formula} (\cite{Ga1}, (3.1)) to write
\begin{align}\label{W2}
\Omega(x) = \Gamma^2(\nu+1)\left( \frac{x}{4}\right)^{-2\nu}\sum_{n=0}^{\infty} C(n, \nu)\frac{(2n+2\nu)}{n+2\nu}\frac{(2\nu+1)_n}{n!}
J_{\nu+n}^2\left( \frac{x}{2}\right)
\end{align}
in which $C(n, \nu)$ denotes the terminating series defined by
\begin{equation}\label{W3}
C(n, \nu)= {}_5F_{4} \left(\begin{array}{c} -n,\,n+2\nu,\,\nu+ 1,\,a, \,a+\frac 12\\
\nu+\frac 12,\,c+1,\,a+b,\,a+b + \frac 12\end{array}\biggl| 1\right)
\end{equation}
and $\nu$ is an arbitrary real number such that $2\nu$ is not a negative integer.

Due to the interlacing property on the zeros of Bessel functions $\,J_\nu,\,J_{\nu+1}\,$
(see Watson \cite{Wa}), the positivity of $\Omega$ would follow instantly from formula \eqref{W2} if
$\,C(n, \nu)>0\,$ for all nonnegative integers $n$ and $\,\nu>-1/2.$

Our investigation on the sign of $C(n, \nu)$ will be carried out along the following steps.
We recall that a ${}_pF_q$ generalized hypergeometric function is said to be {\it  Saalsch\"utzian} if the
sum of numerator parameters plus one is equal to the sum of denominator parameters.

\paragraph{Step 1.} We choose $\,\nu>-1/2\,$ in such a unique way that each coefficient $C(n, \nu)$ becomes
a Saalsch\"utzian terminating series, that is,
\begin{equation}\label{W4}
\nu = b + \frac c2 -\frac 14\quad\text{with}\quad b+ \frac c2  + \frac 14>0.
\end{equation}

\paragraph{Step 2.} In \cite{Ga2}, Gasper discovered a summation formula which states
\begin{align*}
&\quad{}_{p+2}F_{p+1} \left(\begin{array}{c} -n,\,a_1,\,\cdots,\,a_{p+1}\\
b_1,\,\cdots,\,b_{p+1}\end{array}\biggl| 1\right)\\
&\quad =\sum_{k=0}^n \binom nk\frac{(b_1+b_2-a_1-1)_k(b_1-a_1)_k(b_2-a_1)_k}{(b_1+b_2-a_1-1)_{2k}}
\frac{(a_2)_k\cdots(a_{p+1})_k}{(b_1)_k\cdots(b_{p+1})_k}\,\\
&\qquad\times\,\,
{}_{p+1}F_{p} \left(\begin{array}{c} k-n,\,k+a_2,\,\cdots,\,k+a_{p+1}\\
2k+b_1+b_2-a_1,\,k+b_3,\cdots,\,k+b_{p+1}\end{array}\biggl| 1\right).
\end{align*}

An application of this formula gives
\begin{align}\label{W5}
C(n, \nu) &= {}_5F_{4} \biggl(\begin{array}{c} -n,\,n+2b + c-\frac 12,\,
b+ \frac c2 + \frac 34,\,a, \,a+\frac 12\\
b+ \frac c2 + \frac 14,\,c+1,\,a+b,\,a+b + \frac 12\end{array}\biggl| 1\biggr)\nonumber\\
&=\sum_{k=0}^n \binom nk\frac{\left(2a+b-\frac c2-\frac{5}{4}\right)_k\left(a -\frac{c}{2}-\frac{3}{4}\right)_k\left(a-\frac{c}{2}-\frac{1}{4}\right)_k}
{\left(2a + b -\frac{c}{2}-\frac{5}{4}\right)_{2k}}\nonumber\\
&\times\,\,\frac{\left(n+2b+c-\frac{1}{2}\right)_k(a)_k\left(a+\frac{1}{2}\right)_k}
{(a+b)_k(c+1)_k\left(a+b+\frac{1}{2}\right)_k\left(b+\frac{c}{2}+\frac{1}{4}\right)_k}\,A_k(a, b, c),
\end{align}
where $\,A_k(a, b, c)\,$ denotes the Saalsch\"utzian series defined as
\begin{align*}
A_k(a, b, c) = {}_{4}F_{3} \left(\begin{array}{c}
k-n,\,k+n+2b + c-\frac{1}{2},\,k+a,\,k+a+\frac{1}{2}\\
2k+ 2a + b -\frac{c}{2}-\frac{1}{4},\,k+b+ \frac{c}{2}+\frac{1}{4},\,k+c+1\end{array}\biggl| 1\right).
\end{align*}

\paragraph{Step 3.} We next apply Whipple's transformation formula (Bailey \cite{Ba}, 7.2(1)),
\begin{align*}
{}_{4}F_{3} \left(\begin{array}{c}
-m,\,x,\,y,\,z\\
u,\,v,\,w\end{array}\biggl| 1\right) &= \frac{(v-z)_m(w-z)_m}{(v)_m(w)_m}\\
&\times\,\,{}_{4}F_{3} \left(\begin{array}{c}
-m,\,u-x,\,u-y,\,z\\
1-v+z-m,\,1-w+z-m,\,u\end{array}\biggl| 1\right),
\end{align*}
valid if it is Saalsch\"utzian, to decompose further
\begin{align}\label{W6}
&A_k(a, b, c) = \frac{\left(b+\frac{c}{2}+\frac{1}{4}-a\right)_{n-k}(c+1-a)_{n-k}}{\left(k+b+ \frac{c}{2}+\frac{1}{4}\right)_{n-k}(k+c+1)_{n-k}}\,\times\nonumber\\
&{}_{4}F_{3} \left(\begin{array}{c}
k-n,\,k-n+2a-b -\frac{3c}{2} +\frac{1}{4},\,k+a+b -\frac{c}{2}-\frac{3}{4},\,k+a\\
k-n+a-b-\frac{c}{2}+\frac{3}{4},\,k-n + a-c,\,2k+2a+b-\frac{c}{2}-\frac{1}{4}
\end{array}\biggl| 1\right).
\end{align}

\paragraph{Step 4.} From expansion formula \eqref{W5}, it is evident $\,C(n, \nu)>0\,$ if
\begin{equation}\label{W6-1}
2a> -b +\frac c2 +\frac 54,\,\, a \ge\frac{c}{2} +\frac{3}{4},\,
\, b+\frac{c}{2}+\frac{1}{4}>0
\end{equation}
and $\,A_k(a, b, c)>0\,$ for all $k$. By using an elementary inequality
$$\frac{(-m)_j (-m+\alpha)_j}{(-m+\beta)_j(-m+\gamma)_j}>0,\quad j=0, 1,\cdots, m,$$
valid as long as $\,\alpha\le 1,\,\beta<1,\,\gamma<1,\,$ we deduce from \eqref{W6}
$\,A_k(a, b, c)>0\,$ when
\begin{equation}\label{W6-2}
2a\le b +\frac{3c}{2} +\frac{3}{4},\, \,a<b+\frac{c}{2}+\frac{1}{4},\, \,a<c+1.
\end{equation}

\medskip

Combining \eqref{W6-1}, \eqref{W6-2}, we may summarize
what we have proved as follows.

\smallskip

\begin{theorem}\label{theorem3}
For $\,a>0,\,b>0,\,c>-1,\,$ we have
\begin{equation}\label{W8}
\Omega(x)={}_2F_{3} \left(\begin{array}{c} a, \,a+\frac 12\\
c+1,\,a+b,\,a+b + \frac 12\end{array}\biggl| -\frac{\,x^2}{4}\right)>0
\qquad(x>0)
\end{equation}
if $\,a, b, c\,$ satisfy the following conditions simultaneously.
\begin{align*}
\left\{\begin{aligned} &{\,\,\frac c2 + \frac 34\le a <\min\left(c+1,\,b+\frac c2 + \frac 14\right),}\\
&{\,\,-b+\frac c2 +\frac 54<2a \le b+\frac{3c}{2} + \frac 34\,.}\end{aligned}\right.
\end{align*}
\begin{itemize}
\item[\rm(i)] In the boundary case $\,a= b+\frac c2 + \frac 14\,,$ \eqref{W8} also holds true if
\begin{equation*}
\frac 12< b\le\frac c2 + \frac 14\,,\,\,c>\frac 12\,.
\end{equation*}
\item[\rm(ii)] In the boundary case $\,c=a-1,\,$ \eqref{W8} also holds true if
\begin{equation*}
b \ge \max\left[1,\,\frac 12\left(a+\frac 32\right)\right],\,\,(a, b) \ne \left(\frac 12,\,1\right).
\end{equation*}
In the case $\,(a, b)=\left(\frac 12,\,1\right),\,$ we have $\,\Omega(x)\ge 0\,$ for $\,x>0.$
\end{itemize}
\end{theorem}

\begin{proof}
It remains to prove positivity in the two boundary cases.

(i)We apply Whipple's
transformation formula of Step 3 directly to obtain
\begin{align*}
C(n, \nu) &= \frac{\left(b-\frac 12\right)_n (b)_n}{\left(2b+\frac c2+\frac 14\right)_n \left(2b+\frac c2+\frac 34\right)_n}\\
&\quad\times\,\,{}_{4}F_{3} \left(\begin{array}{c}
-n,\,-n -2b + \frac 32,\,-b+ \frac c2+\frac 14,\,b+\frac c2+\frac 34\\
-n-b + \frac 32,\,-n -b+1,\,c+1
\end{array}\biggl| 1\right),
\end{align*}
which is easily seen to be positive under the former condition by the same reasonings as in Step 4.
If $\,b = c/2 + 1/4,\,a= c+1/2,\,$ then $\Omega$ reduces to
$$\Omega(x)={}_1F_{2} \left(\begin{array}{c} c+\frac 12\\
\frac 32\left(c+\frac 12\right),\,\frac 32\left(c+\frac 12\right) +\frac 12\end{array}
\biggl| -\frac{\,x^2}{4}\right)$$
and positivity with $\,c>1/2\,$ follows by Lemma \ref{lemma1} (see also Fields and Ismail \cite{FI}).

(ii) In this case, it is easy to deduce again from Lemma \ref{lemma1}
$$\Omega(x)={}_1F_{2} \left(\begin{array}{c} a+\frac 12\\
a+b,\,a+b + \frac 12\end{array}\biggl| -\frac{\,x^2}{4}\right)>0$$
under the stated condition.
In the case $\,(a, b)=\left(\frac 12,\,1\right),\,$ $\Omega$ reduces to
$$\Omega(x) = \left[\frac{\sin(x/2)}{x/2}\right]^2\ge 0.$$
\end{proof}

In the special case $\,b=1,\,$ we obtain

\smallskip

\begin{corollary}\label{corollary1}
For $\,a>0,\,c>-1,\,$ we have
\begin{equation}\label{W9}
{}_2F_{3} \left(\begin{array}{c} a, \,a+\frac 12\\
c+1,\,a+1,\,a+ \frac 32\end{array}\biggl| -\frac{\,x^2}{4}\right)>0
\qquad(x>0)
\end{equation}
if $\,a, c\,$ satisfy one of the following conditions.
\begin{align*}
&{\rm(i)}\quad \left\{\begin{aligned} &{\,\,\frac c2 + \frac 34\le a<\min\left(c+1,\,\frac c2 + \frac 54\right),}\\
&{\,\,\frac c2 +\frac 14<2a\le \frac{3c}{2} + \frac 74\,.}\end{aligned}\right.\\
&{\rm(ii)}\quad a= \frac c2 + \frac 54\,,\,\,c\ge \frac 32\,.\quad {\rm(iii)}\quad c= a-1,\,\,0<a\le \frac 12\,.
\end{align*}
\end{corollary}

\section{Improved results of Misiewicz and Richards}
In the case $\,0<\mu\le 1\le\lambda,\,$ the density
$\,t\mapsto (1-t^\mu)_+^\lambda\,$ is convex and non-increasing on $(0, \infty).$
By using Williamson's characterization \cite{Wi} on such monotone convex functions,
Misiewicz and Richards \cite{MR} observed
\begin{align}\label{M0}
&\int_0^x \left(x^\mu -t^\mu\right)^\lambda t^\alpha J_\beta(t) dt = x^{\mu\lambda-1}\int_0^1 K(xt)\,dG(t),\nonumber\\
&\qquad\qquad K(x)=\int_0^x (x-t) t^\alpha J_\beta(t) dt,
\end{align}
with a unique probability measure $G$,
so that the positivity of \eqref{A1} under consideration
would follow once kernel $K$ were shown to be nonnegative.

In view of the well-known Bessel identity (Watson \cite{Wa})
\begin{equation}\label{J1}
J_\beta(x) = \frac{1}{\Gamma(\beta+1)}\left(\frac x2\right)^\beta\,{}_0F_1\left(\beta+1\,; -\frac{\,x^2}{4}\right),\quad\beta>-1,
\end{equation}
it is simple to modify \eqref{W1} to evaluate
\begin{align}\label{M2}
K(x) = \frac{B(\alpha+\beta+1, 2) x^{\alpha+\beta+2}}{2^\beta\,\Gamma(\beta+1)}\,
{}_2F_{3} \left(\begin{array}{c} \frac{\alpha+\beta+1}{2}\,, \,\frac{\alpha+\beta+2}{2}\\
\beta +1,\,\frac{\alpha+\beta+3}{2},\,\frac{\alpha+\beta+4}{2}\end{array}\biggl| -\frac{\,x^2}{4}\right)
\end{align}
and hence the problem of positivity reduces to the nonnegativity question on the ${}_2F_3$ hypergeometric functions
defined in \eqref{M2}.

Our improvement of Theorem A reads as follows.

\smallskip

\begin{theorem}\label{theorem4}
Let $\mathcal{P}$ be the set of parameters $(\beta, \alpha)$ defined by
\begin{align*}
\mathcal{P} = \left\{\beta>-1,\,\,-\beta-1<\alpha\le\min \left[
\beta+1,\, \frac 12\left(\beta+\frac 32\right), \,\frac 32\right]\right\}.
\end{align*}
For $\,0<\mu\le 1\le\lambda\,$ and $\,(\beta, \alpha)\in\mathcal{P},\,$ we have
\begin{equation*}
\int_0^x \left(x^\mu-t^\mu\right)^\lambda t^\alpha J_\beta(t)dt>0 \qquad(x>0).
\end{equation*}
\end{theorem}

\smallskip

\begin{remark}
In Figure \ref{Fig4}, the trapezoid with vertices
$$\,(-1, 0), \,(-1/2, -1/2),\,(3/2, 3/2), \,(-1/2, 1/2)\,$$
is newly added to the positivity region of Misiewicz and Richards which corresponds to
the infinite polygon bounded by
$\,\alpha=-\beta-1,\,\,\alpha=\beta,\,\,\alpha=\frac 32\,.$

\begin{figure}[!h]
 \centering
 \includegraphics[width=280pt, height=220pt]{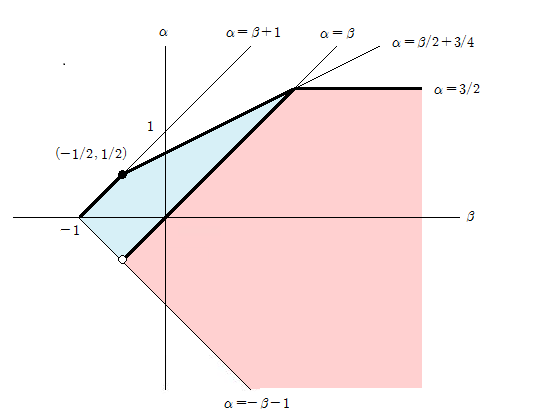}
 \caption{The positivity region in the case $\,0<\mu\le 1\le\lambda\,$ which improves the ones of Misiewicz and Richards (pink)
 and Kuttner (black dot).}
\label{Fig4}
\end{figure}

\end{remark}

\begin{proof}
The ${}_2F_3$ hypergeometric functions of \eqref{M2} are of type \eqref{W9}
with
$$a= \frac{\alpha+\beta+1}{2}\,,\quad c=\beta.$$
It is simple to find condition (i) of Corollary \ref{corollary1}
gives the infinite strip
\begin{equation}\label{M3}
\beta\ge -\frac 12\,,\quad \frac 12\le \alpha<\frac 32\,,\quad \alpha\le \frac 12\left(\beta+\frac 32\right)
\end{equation}
as a positivity region. Likewise, conditions (ii), (iii) correspond to the boundary lines
\begin{equation}\label{M4}
\beta\ge \frac 32\,,\, \alpha=\frac 32\quad\text{and}\quad -1<\beta\le -\frac 12\,,\,\alpha = \beta+1.
\end{equation}

To fill out the remaining positivity region, we observe from Lemma \ref{lemma3} that
if \eqref{M2} is positive with some parameters $\,\alpha_0, \beta_0,\,$ then positivity continues to hold true
for all parameters $\,\alpha_0-\delta,\,\beta_0+\delta,\,\delta\ge 0,\,$ that is, for all $\,\alpha, \beta\,$
lying on the half-line emanating from $(\beta_0,\alpha_0)$ defined by $\,\alpha = -\beta + \alpha_0 +\beta_0,\,\beta\ge \beta_0.\,$
By adding all half-lines emanating from $\,(\beta, \alpha)\in\mathcal{P}\,$ constructed from \eqref{M3}, \eqref{M4},
we obtain the full stated region $\mathcal{P}$.
\end{proof}

\smallskip

In the special case $\,\mu=1,\,$ reduction \eqref{M0} is unnecessary for
\begin{align}\label{M6}
&\int_0^x \left(x -t\right)^\lambda t^\alpha J_\beta(t) dt =\frac{B(\alpha+\beta+1, \lambda+1) x^{\lambda+\alpha+\beta+1}}
{2^\beta\,\Gamma(\beta+1)}\nonumber\\
&\qquad\qquad\quad\times\,\,
{}_2F_{3} \left(\begin{array}{c} \frac{\alpha+\beta+1}{2}\,, \,\frac{\alpha+\beta+2}{2}\\
\beta +1,\,\frac{\alpha+\beta+\lambda+2}{2},\,\frac{\alpha+\beta+\lambda+3}{2}\end{array}\biggl| -\frac{\,x^2}{4}\right)
\end{align}
in which the generalized hypergeometric functions are of type \eqref{W8} with
$$ a= \frac{\alpha+\beta+1}{2}\,,\,\,b= \frac{\lambda+1}{2}\,,\,\,c=\beta.$$

A direct application of Theorem \ref{theorem3} yields the following improved result.
\smallskip

\begin{theorem}\label{theorem5}
For $\,\lambda>0\,$ and $\,(\beta, \alpha)\in\mathcal{O},\,$ we have
\begin{equation*}
\int_0^x \left(x-t\right)^\lambda t^\alpha J_\beta(t)dt>0 \qquad(x>0),
\end{equation*}
where $\mathcal{O}$ denotes the set of parameters defined by
\begin{align*}
\mathcal{O} = \left\{\begin{aligned} &{\quad\beta>-1,\,\,-\beta+1-\lambda<\alpha\le\min \left[
\frac 12\left(\beta+\lambda +\frac 12\right), \,\lambda+\frac 12\right],} &{\text{if}\quad\lambda<1,}\\
&{\beta>-1,\,\,-\beta-1<\alpha\le\min \left[
\beta+1,\, \frac 12\left(\beta+\lambda +\frac 12\right), \,\lambda+\frac 12\right],} &{\text{if}\quad\lambda\ge 1.}
\end{aligned}\right.
\end{align*}
\end{theorem}

As the proof proceeds in the same fashion as above, we shall omit it. We
refer to Gasper \cite{Ga1}, \cite{Ga2} for related results and further applications.

In view of the identity
$$J_{-\frac 12}(t) = \sqrt{\frac {2}{\pi\,t}}\,\cos t,$$
the case $\,\beta = -1/2\,$ in Theorem \ref{theorem1}, Theorem \ref{theorem4}
corresponds to Kuttner's result.

\smallskip

\begin{corollary}
If $\,\mu>0, \,\lambda\ge 0,\,-1<\alpha\le -\frac 12\,$ or $\,0<\mu\le 1\le\lambda,\, -1<\alpha\le 0,\,$ then
\begin{equation*}
\int_0^x \left(x^\mu -t^\mu\right)^\lambda t^{\alpha} \cos t\,dt>0 \qquad(x>0).
\end{equation*}
\end{corollary}

In the case $\,\alpha=0,\,$ the problem of determining the positivity range of $\,\lambda = \lambda(\mu)\,$ is a long-standing
open problem and we refer to Golubov \cite{Go} and Gneiting, Konis and Richards \cite{GKR} for
partial results and further references.

\section{Buhmann's radial basis functions}
While studying scattered data approximations, Buhmann \cite{Bu} introduced a 4-parameter family of
compactly supported {\it radial basis functions}
on $\R^n$ defined as follows.

\smallskip

\begin{definition}(Buhmann's radial basis functions)\\
For $\,\delta>0,\,\rho\ge 0,\,\lambda>-1\,$ and $\,\alpha>- n/2-1,\,$ define
\begin{equation}\label{B1}
W(\bx)=\int_{0}^{\infty} \left(1-\frac{\|\bx\|^2}{t}\right)_{+}^{\lambda}
t^{\alpha}\left(1-t^{\delta}\right)_{+}^{\rho}\,dt\qquad(\bx\in\R^n),
\end{equation}
where $\|\bx\|$ stands for the Euclidean norm $\,\|\bx\|^2= \bx\cdot\bx.$
\end{definition}

For this family of compactly supported radial functions,
Buhmann proved the following {\it positive definiteness} (see also \cite{Z}).

\smallskip

\begin{itemize}
\item[{}] {\bf Theorem C.} (Buhmann, \cite{Bu}) {\it Let $\mathcal{B}_{n}$ be the set of parameters $(\lambda, \alpha)$ defined by
\begin{align*}
\mathcal{B}_{1}&=\left\{\lambda>-\frac{1}{2},\,
-1<\alpha\le\min\left(\lambda-\frac{1}{2},\,\frac \lambda 2\right)\right\},\\
\mathcal{B}_{2}&=\left\{\lambda>-\frac{1}{2},\,
-1<\alpha\le\min\left[\frac{1}{2}\left(\lambda-\frac{1}{2}\right), \lambda-\frac{1}{2}\right]\right\},\\
\mathcal{B}_{3}&=\left\{\lambda\ge0,\, -1<\alpha\le\frac{1}{2}\left(\lambda-1\right)\right\},\\
\mathcal{B}_{n}&=\left\{\lambda>\frac{n-5}{2},\, -1<\alpha\le\frac{1}{2}
\left(\lambda-\frac{n-1}{2}\right)\right\} \quad \text{if} \quad n\ge 4.
\end{align*}
For $\,0<\delta\le\frac{1}{2}\,,\,\rho\ge1,\,$ if $\,(\lambda, \alpha)\in\mathcal{B}_{n},\,$ then
$W$ has a strictly positive Fourier transform
and hence induces positive definite matrices on $\R^{n}$.}
\end{itemize}

As Buhmann calculated, the Fourier transform of $W$ is $\,\widehat{W}(\xi) = \omega(\|\xi\|),\,\xi\in\R^n,\,$
where
\begin{equation}\label{B2}
\omega(x) =\frac{(2\pi)^{\frac n 2}2^{\lambda+1}\Gamma(\lambda+1)}{x^{n+2+2\delta\rho+2\alpha}}
\int_{0}^{x}\left(x^{2\delta}-t^{2\delta}\right)^{\rho}t^{2\alpha+1-\lambda+\frac{n}{2}}J_{\lambda+\frac{n}{2}}(t)\,dt
\end{equation}
for $\,x>0\,$ and Buhmann exploited Theorem A to establish the above result.

Our purpose here is to extend $\mathcal{B}_n$ in several directions.

We begin with extending Theorem C with the aid of Theorem \ref{theorem4}.

\smallskip

\begin{theorem}\label{theorem6}
Let $\mathcal{P}_{n}$ be the set of parameters $(\lambda, \alpha)$ defined by
\begin{equation*}
\mathcal{P}_{n}=\left\{\lambda>-1,\, -\frac{n+2}{2}<\alpha\le\min\left[\frac{1}{4}\left(3\lambda-\frac{n+1}{2}\right),\,\frac{1}{2}\left(\lambda-\frac{n-1}{2}\right)\right]\right\}.
\end{equation*}
For $\,0<\delta\le\frac{1}{2}\,,\,\rho\ge1,\,$ if $\,(\lambda, \alpha)\in\mathcal{P}_{n},\,$ then
$W$ has a strictly positive Fourier transform
and hence induces positive definite matrices on $\R^{n}$.
\end{theorem}

\smallskip

\begin{figure}[!h]
 \centering
 \includegraphics[width=300pt, height=220pt]{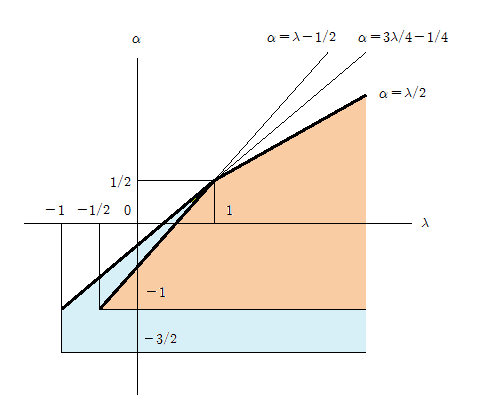}
 \caption{The regions of positive definiteness $\mathcal{B}_1$ (yellow) and $\mathcal{P}_1$.}
\label{Fig5}
\end{figure}

\medskip

\begin{remark} The proof follows trivially upon renaming parameters, that is,
$$2\delta\to \mu,\,\,\rho\to\lambda, \,\,2\alpha+1-\lambda+n/2\to\alpha,\,\,\lambda+n/2\to\beta.$$
It is easy to observe $\,\mathcal{B}_n\subset\mathcal{P}_n\,$ in a proper way,
as shown in Figure \ref{Fig5} in the one-dimensional case. In addition, an inspection on the two boundary lines
of $\mathcal{P}_n$ reveals
\begin{equation*}
\mathcal{P}_{n}=\left\{\lambda>-1,\, -\frac{n+2}{2}<\alpha\le\frac{1}{2}\left(\lambda-\frac{n-1}{2}\right)\right\}
\quad\text{for}\quad n\ge 5.
\end{equation*}
\end{remark}

\smallskip

In the case $\,\delta=1/2,\,$ Theorem \ref{theorem5} gives an improvement.
\smallskip

\begin{theorem}\label{theorem7}
For $\,\delta=\frac 12\,,\,\rho\ge 1,\,$ Theorem \ref{theorem6} continues to hold true if $\mathcal{P}_n$
is replaced by the set $\mathcal{O}_n$ of parameter pairs $(\lambda, \alpha)$ defined as
\begin{equation*}
\mathcal{O}_n = \left\{\lambda>-1,\,-\frac{n+2}{2}<\alpha\le\min\left[\lambda,\,\frac{1}{4}\left(3\lambda-\frac{n+3}{2}+\rho\right),\,
\frac{1}{2}\left(\lambda-\frac{n+1}{2}+\rho\right)\right]\right\}.
\end{equation*}
\end{theorem}

\smallskip

\begin{remark} In the special occasion $\,\delta= \frac 12,\, \rho=\frac{n+1}{2} +\sigma,\,\lambda=\alpha\,$ with $\,\sigma\ge 0,\,$
Buhmann's radial basis functions take the form
\begin{equation}
W(\bx) = 2\int_{\|\bx\|}^1\left(t^2-\|\bx\|^2\right)^{\alpha} t\left(1-t\right)^{\frac{n+1}{2}+\sigma}\, dt,
\end{equation}
known as Wendland's functions (see \cite{S}, \cite{We1}, \cite{We2}), which are are easily seen to be positive definite for $\,-1<\alpha\le \sigma-1\,$
according to the above theorem. Obviously, Buhmann's original theorem, Theorem C, is not applicable in this case.
\end{remark}

\smallskip

In the unrestricted case $\,\delta>0, \,\rho\ge 0\,$, Theorem \ref{theorem1} yields the following.

\smallskip

\begin{theorem}\label{theorem8}
Let $\mathcal{R}_{n}$ be the set of parameters $(\lambda, \alpha)$ defined by
\begin{align*}
\mathcal{R}_{1}&=\left\{\lambda>-1,\, -\frac{3}{2}<\alpha\le\frac{1}{2}\left(\lambda-\frac{3}{2}\right)\right\}\\
&\qquad\cup\,\left\{\lambda>-\frac{1}{2},\,\frac{1}{2}\left(\lambda-\frac{3}{2}\right)<\alpha\le
\min\left[\lambda-\frac{1}{2},\,\frac{1}{2}\left(\lambda-1\right)\right]\right\}\,,\\
\mathcal{R}_{n}&=\left\{\lambda>-1,\, -\frac{n+2}{2}<\alpha\le\min\left[\lambda-\frac{1}{2},\,
\frac{1}{2}\left(\lambda-\frac{n+1}{2}\right)\right]\right\},\quad n\ge 2.
\end{align*}
For $\,\delta>0, \,\rho\ge 0,\,$ if $\,(\lambda, \alpha)\in\mathcal{R}_{n},\,$ then each
$W$ has a nonnegative non-vanishing Fourier transform and hence induces positive definite matrices on $\R^{n}$.
\end{theorem}

\smallskip

\begin{remark} The Fourier transform of $W$ is indeed strictly positive unless
\begin{equation}\label{B3}
\rho=0,\,\alpha=-\frac{n}{2}\,,\,\lambda=-\frac{n-1}{2}\,,\, n= 1, 2.
\end{equation}
In such an exceptional case of \eqref{B3}, it is simple to evaluate
\begin{align}
W(\bx) = \left\{\begin{aligned} &{2\left( 1- \|\bx\|\right)} &{\text{if}\quad n=1,}\\
&{2\ln\left(\frac{1 +\sqrt{1-\|\bx\|^2\,}}{\|\bx\|}\right)} &{\text{if}\quad n=2,}\end{aligned}\right.
\end{align}
for $\,\|\bx\|\le 1\,$ and zero otherwise. Moreover, its Fourier transform is given by
\begin{equation}
\widehat{W}(\xi) = 2\pi^{\frac{n-1}{2}} \Gamma\left(\frac{3-n}{2}\right)\left[\frac{\sin(\|\xi\|/2)}
{\|\xi\|/2}\right]^2\ge 0\qquad(\xi\in\R^n).
\end{equation}
\end{remark}

\smallskip

An improvement in the case $\,\delta =1\,$ owes to Theorem \ref{theorem2}.

\smallskip

\begin{theorem}
For $\,\delta=1, \,\rho\ge 0,\,$ Theorem \ref{theorem8} continues to hold true if $\mathcal{R}_n$ is replaced by
the set $\mathcal{S}_n$ of parameter pairs $(\lambda, \alpha)$ defined as
\begin{align*}
\mathcal{S}_{1}&=\left\{\lambda>-1,\, -\frac{3}{2}<\alpha\le\frac{1}{2}\left(\lambda-\frac{3}{2}\right)\right\}\\
&\quad\cup\,\left\{\lambda>-\frac{1}{2},\,\frac{1}{2}\left(\lambda-\frac{3}{2}\right)<\alpha\le
\min\left[\lambda-\frac{1}{2},\,\frac{1}{2}\left(\lambda-1+\rho\right)\right]\right\},\\
\mathcal{S}_{n}&=\left\{\lambda>-1,\, -\frac{n+2}{2}<\alpha\le\min\left[\lambda-\frac{1}{2},\,
\frac{1}{2}\left(\lambda-\frac{n+1}{2}+\rho\right)\right]\right\},\quad n\ge 2.
\end{align*}
\end{theorem}

\smallskip

\begin{remark}
The Fourier transform of $W$ is strictly positive unless
\begin{equation}
\alpha=\rho-\frac{n}{2}\,,\,\lambda=\rho-\frac{n-1}{2}\,,\, 1\le n\le\left[2\rho+3\right].
\end{equation}
In such an exceptional case, $\,W,\,\widehat{W}\,$ are explicitly given by
\begin{align}
\left\{\begin{aligned} &{
W(\bx) = 2\int_{\|\bx\|}^1\left(t^2-\|\bx\|^2\right)^{\rho+\frac 12-\frac n2} \left(1-t^2\right)^\rho\, dt,}\\
&{\widehat{W}(\xi) = \frac{\pi^{\frac{n+1}{2}}\Gamma(\rho+1)\Gamma\left(\frac{2\rho+3-n}{2}\right)}{
(\|\xi\|/2)^{2\rho+1}}
J_{\rho+\frac 12}^2\left(\frac{\|\xi\|}{2}\right)}\end{aligned}\right.
\end{align}
for $\,\|\bx\|\le 1\,$ and $\,\xi\in\R^n.$ In the special case $\,\rho = \frac{n-1}{2}\,,$ radial basis function
$W$ is often referred to as Euclid's hat function (see Gneiting \cite{Gn}).
\end{remark}

\section{Appendix: Newton diagram of positivity}
We shall assume $\,x>0\,$ in what follows and write
\begin{align}\label{N1}
\J_\nu(x) = {}_0F_1\left(\nu +1\,; -\frac{\,x^2}{4}\right)\qquad(\nu>-1).
\end{align}
In view of \eqref{J1}, it is evident that the $\J_\nu$ share positive zeros in common with
Bessel functions $J_\nu$. A basic principle of
positivity is the following analogue of Lemma \ref{lemma3} for $\F$ hypergeometric functions
which can be proved in the same manner.

\smallskip

\begin{lemma}\label{lemmaA} For $\,a>0,\,b>0,\,c>0,\,$ suppose that
$$\F\left(a\,; b, c\,; -\frac{\,x^2}{4}\right)\ge 0.$$
Then for any $\,0\le\gamma<a,\,\delta\ge 0,\,\epsilon\ge 0,\,$
not simultaneously zero,
$$\,\F\left(a-\gamma\,;b+\delta, \,c+\epsilon; -\frac{\,x^2}{4}\right)>0.$$
\end{lemma}

\paragraph{Proof of Lemma \ref{lemma1}.}
For part (i), if $\,\phi\ge 0\,$ and $\,0<b\le a,\,$ then Lemma \ref{lemmaA} implies
$$\F\left(b\,;\,b, c\,; -\frac{\,x^2}{4}\right) = \J_{c-1}(x)>0,$$
which contradicts the fact $\J_{c-1}$
has infinitely many positive zeros. Thus $\,b>a\,$ and
$\,c>a\,$ by symmetry. In view of the asymptotic behavior (\cite{L})
\begin{align}\label{I3}
\phi(x) &=\frac{\Gamma(b)\Gamma(c)}{\Gamma(b-a)\Gamma(c-a)}
\,\left(\frac x2\right)^{-2a} \left[ 1+ O\left(x^{-2}\right)\right]\nonumber\\
& + \,\,\frac{\Gamma(b)\Gamma(c)}{\sqrt \pi\,\Gamma(a)}
\,\left(\frac x2\right)^{-\sigma} \left[\cos\left( x - \frac{\pi\sigma}{2}\right)
+ O\left(x^{-1}\right)\right]
\end{align}
as $\,x\to\infty,\,$ where $\,\sigma = b+c -a - 1/2,\,$
it is immediate to observe
the condition $\, \sigma\ge 2a\,$ is necessary, that is, $\,b+c\ge 3a + 1/2\,.$

Regarding part (ii), observe first (\cite{Wa}, Chapter 5) that
$$\F\left(a\,; a+\frac 12, \,2a\,; -\frac{\,x^2}{4}\right) = \J_{a-\frac 12}^2\left(\frac x2\right)\ge 0.$$
For $\,a=1/2,\,$ the two points of $\Lambda$ coincide and the positivity of $\phi$
with parameter pair $\,(b, c)\in P_{1/2}\,$ follows from Lemma \ref{lemmaA}. For $\,a\ne 1/2,\,$
if $(b, c)$ lies on the boundary line of $P_a$, that is, $\,c= 3a+1/2-b,\,$ then it is not hard
to compute the coefficients of Gasper's sums of squares series expansion by Saalsch\"utz's formula to deduce
\begin{align}\label{G7}
&\F\left(a\,; b, \, 3a + \frac 12 -b \,; -\frac{\,x^2}{4}\right)
=\Gamma^2\left(a+\frac 12\right)\left(\frac x4\right)^{-2a-1}\nonumber\\
&\quad\times\,\, \sum_{n=0}^\infty \frac{2n + 2a-1}{n+ 2a-1}\frac{(2a)_n}{n!}
\frac{(2a-b)_n \left(b-a- 1/2\right)_n}{(b)_n \left(3a+1/2-b\right)_n}
J^2_{n+ a-\frac 12}\left(\frac x2\right),
\end{align}
which is easily seen to be positive when $b$ lies strictly between $a+1/2$ and $2a$. The positivity
of $\phi$ for $\,(b, c)\in P_a\,$ now follows from this boundary
case and Lemma \ref{lemmaA}.
\qed

\bigskip

{{\large \bf Acknowledgements.}} Yong-Kum Cho is supported by the National Research Foundation of Korea Grant
funded by the Korean Government \# 20160925. Seok-Young Chung is supported by the Chung-Ang University
Excellent Student Scholarship in 2017. Hera Yun is supported by the Chung-Ang University
Research Scholarship Grants in 2014-2015.

\end{document}